\documentclass[a4paper,12pt]{article}
\usepackage{latexsym}
\usepackage{amsmath,amsthm,amsfonts,amssymb}
\usepackage{dsfont}
\usepackage{upgreek}
\usepackage{tipa}
\usepackage{graphicx}
\usepackage{enumitem} 
\usepackage[square,sort,comma,numbers]{natbib}
\usepackage[utf8]{inputenc}
\usepackage{float}
\newtheorem{theorem}{Theorem}[section]
\newtheorem{corollary}[theorem] {Corollary}

\newtheorem{definition}[theorem]{Definition}

\newtheorem{lemma} [theorem]{Lemma}

\theoremstyle{remark}
\newtheorem{remark}[theorem]{Remark}
\vspace{5cm}
\begin{document}
		
	\begin{center}
		{{\Large \textbf {Fractional differ-integral involving bicomplex
Prabhakar function in the kernel and applications 
}}
			\medskip

	{\sc  Urvashi Purohit Sharma$^{1},$\ \ Ritu Agarwal$^{2*}$ }\\
	{\footnotesize $^{1}$Department of Science and Humanities, Kalaniketan Polytechnic College Jabalpur-482001, INDIA}\\
	{  \footnotesize $^{2}$Department of Mathematics,
		Malaviya National Institute of Technology, Jaipur-302017, INDIA}
}\\
{\footnotesize E-mail: {\it $^{1}$urvashius100@gmail.com,  $^{2} $ragarwal.maths@mnit.ac.in }
	
}        \end{center}
\thispagestyle{empty}
	
\hrulefill
\begin{abstract}
\indent

This paper introduces the bicomplex Prabhakar derivative, extending fractional calculus to four-dimensional bicomplex spaces.  
Using the generalized kernel involving bicomplex Prabhakar function, we construct the bicomplex Prabhakar derivative and prove fundamental operational properties including linearity, composition rules, and connections to Riemann-Liouville and Caputo operators. We further investigate how fractional operators act on the bicomplex Prabhakar function itself, developing integral representations and transformation formulas. 

This work provides a rigorous foundation for modeling complex phenomena with memory effects and multi-dimensional coupling in bicomplex domains. The rich algebraic structure of bicomplex numbers, combined with the flexibility of Prabhakar kernels, offers a versatile framework applicable across diverse scientific and engineering disciplines.

\end{abstract}

\hrulefill
	


*Corresponding author [Correspondence to ragarwal.maths@mnit.ac.in]
	
\section{Introduction}
Bicomplex fractional calculus extends classical fractional calculus \cite{kilbas2006} to four-dimensional bicomplex spaces, combining the advantages of fractional operators (modeling memory effects and non-local dynamics) with the richer structure of commutative bicomplex analysis \cite{price1991}. This framework finds compelling applications across diverse domains: electromagnetic theory (fractional Maxwell equations in anisotropic media \cite{kumar2023generalization}), heat transfer and anomalous diffusion, population dynamics (coupled predator-prey models with hereditary effects), electric circuit analysis with coupled fractional elements \cite{thirumalai2025perception}, and quantum mechanics (time-fractional Schrödinger equations). The Prabhakar derivative, characterized by three-parameter Mittag-Leffler functions \cite{prabhakar1971singular}, offers greater flexibility than Riemann-Liouville or Caputo operators through an additional parameter enabling better experimental data fitting and regularization properties. Despite recent progress in bicomplex Riemann-Liouville operators \cite{goswami2022riemann}, bicomplex Prabhakar derivatives remain unexplored, motivating this work to establish their theoretical foundations and investigate applications in coupled multi-dimensional systems exhibiting memory and non-local interactions.

Bicomplex numbers were introduced in the late 19th century by Corrado Segre \cite{csegre1892} as part of the broader study of hypercomplex systems. They form a commutative algebra over the real field, unlike quaternions, which are non-commutative. Bicomplex numbers provide a natural framework for handling problems involving multiple complex variables while preserving commutativity, making them an elegant generalization of ordinary complex numbers. A number of features of the bicomplex numbers have been found.  The various algebraic and geometric characteristics of bicomplex numbers and their applications have been the focus of research in recent years (see,  e.g. \cite{  rochon2004b, ronn2001, meluna2012}).

\subsection{Bicomplex Numbers}
Segre \cite{csegre1892} defined the set of bicomplex numbers as:
\begin{definition}[Bicomplex Number]
In terms of real components, the set of bicomplex numbers is defined as
\begin{equation}
\mathbb{T} =\{\xi: \xi = x_0 +i_1 x_1 +i_2 x_2 + j x_3~| ~x_0,~x_1,~x_2,~x_3~ \in \mathbb{R} \}, 	
\end{equation}
and in terms of complex numbers it can be written as
\begin{equation}\label{eq:bc}
\mathbb{T} = \{\xi: \xi =z_1 + i_2z_2~ |~z_1,~z_2 ~\in \mathbb{C}\}.	
\end{equation} 
\end{definition}
We shall use the  notations: $i_1^2=-1, ~i_2^2=-1,~ j=i_1i_2,~ j^2=1$, $x_0=\operatorname{Re}(\xi),~x_1=\operatorname{Im}_{i_1}(\xi),~x_2=\operatorname{Im}_{i_2}(\xi),~x_3=\operatorname{Im}_{j}(\xi).$  Here, $i_1,i_2$ are imaginary units while $j$ is the hyperbolic unit.

Two non trivial idempotent zero divisors in $\mathbb{T},$ denoted by $e_1$ and $e_2$ and defined as follows \cite{price1991}: \\
$	e_1 =  \dfrac{1 + i_1i_2}{2}=\dfrac{1+j}{2},~ e_2   =\dfrac  {1 - i_1i_2}{2} =\dfrac{1-j}{2}$, $e_1.e_2   = 0$,  $ e_1+ e_2 = 1$	   and $e_1^2=e_1,~e_2^2=e_2$.

In terms of $e_1$ and $e_2$, each element $\xi\in \mathbb{T}$ has an unique idempotent representation that is specified by
\begin{equation}\label{eq:idempotent def}
\xi = z_1 + i_2 z_2   =  \xi_1 e_1 + \xi_2 e_2,
\end{equation}  where  $\xi_1 = (z_1 - i_1 z_2)$ and $\xi_2 = (z_1 + i_1 z_2).$ \\

\begin{definition}[Bicomplex Partial Order]
\label{def:bicomplex_order}
For $\xi, \eta \in T$ with idempotent representations 
$\xi = \xi_1 e_1 + \xi_2 e_2$ and $\eta = \eta_1 e_1 + \eta_2 e_2$, 
we write 
\[
\xi \prec \eta \quad \text{if and only if} \quad |\xi_1| < |\eta_1| \text{ and } |\xi_2| < |\eta_2|.
\]
\end{definition}

Projection mappings  $P_1 : 	\mathbb{T} \rightarrow T_1 \subseteq \mathbb{C},$ $ P_2 : \mathbb{T} \rightarrow T_2 \subseteq \mathbb{C} $ for a bicomplex number $\xi=z_1+i_2z_2$ are defined as  (see, e.g. \cite{rochon2004b}):\\
\begin{equation}\label{eq:p1}
P_1(\xi)=	\xi_1 \in T_1~\text{and }~P_2(\xi)=\xi_2 \in T_2,
\end{equation}

where
\begin{equation}\label{eq:  a space A1}
T_1= \{\xi_1= z_1-i_1z_2 \hspace{1mm}| z_1,z_2 \in \mathbb{C}\}~\text{and}~ 	T_2= \{\xi_2= z_1+i_1z_2\hspace{1mm} | z_1,z_2 \in \mathbb{C}\}.
\end{equation}
	The present literature is limited to fundamental aspects of bicomplex numbers. 
    Other properties can be found in \cite{price1991,ronn2001,riley1953,meluna2012}.

		Many authors have  worked on different generalization of Mittag-Leffler function (ML function) their properties and application of function in the area of fractional calculus. These functions formed as a result of the development of the fractional calculus and the solution of new forms of integral and differential equations \cite{gorenflo2014}.\\	

In recent research, efforts have been made to extend   Mittag-Leffler type function and applications \cite{agarwal2023bicomplex,ragarwal2022,ragarwal2022sept,ragarwal2021sept,ragarwal2023solution,ragarwal2025} in the bicomplex space. 
  Agarwal et al. \cite{ragarwal2022} introduce bicomplex one parameter ML functionas as 
\begin{equation}\label{eq :ml1}
\mathbb{E}_{m}(\zeta)=\sum_{u=0}^{\infty}\frac{\zeta^u}{\Gamma_2(m u+1 )}, \end{equation}
where $|\operatorname{Im_j}(m)|<\operatorname{Re}(m).$
Here $\Gamma_2$ is the bicomplex Gamma function \cite{goyal2006}. 


Bicomplex two parameter ML function defined by Sharma et al. \cite{ragarwal2021sept}
\begin{equation}\label{eq :ml2}
\mathbb{E}_{m,n}(\zeta)=\sum_{u=0}^{\infty}\frac{\zeta^u}{\Gamma_2(m u+n )},  \end{equation}
where $|\operatorname{Im_j}(m)|<\operatorname{Re}(m),~|\operatorname{Im_j}(n)|<\operatorname{Re}(n).$

Bicomplex one parameter k-ML function defined by Bakhet et al \cite{bakhet2025new} as follows
\begin{equation}
\mathbb{E}_{k,m}(\zeta) = \sum_{u=0}^{\infty} 
\frac{\zeta^{u}}{\Gamma_{2,k}(mu + 1)},\quad \zeta,m\in \mathbb{T},
\label{eq:kml}
\end{equation}
where $|\operatorname{Im_j}(m)|<\operatorname{Re}(m).$
Here $\Gamma_{2,k}$ is k-bicomplex Gamma function \cite{zayed2025k} defined as
\begin{equation}\label{k gamma}
\Gamma_{2,k}(\zeta) = \int_{\Upsilon} 
e^{-\frac{t^{k}}{k}}  t^{\zeta - 1}  dt,
\end{equation}
where $\Upsilon = (\gamma_{1}, \gamma_{2}), 
\text{with } \gamma_{1} \equiv \gamma_{1}(t_{1}), \, 
\gamma_{2} \equiv \gamma_{2}(t_{2})$ and $ t_{1}: 0 ~\text{to}~ \infty,~t_{2}: 0 ~~\text{to}~ \infty.$ \\
By substituting $k=1$ in the  equation \eqref{k gamma}, we get bicomplex Gamma function  defined by Goyal et al \cite{goyal2006}. 

 Bakhet et al \cite{bakhet2025bicomplex} introduced the two parameter k-ML  function as follows
\begin{equation}
\mathbb{E}_{k,m,n}(\zeta) = \sum_{u=0}^{\infty} 
\frac{\zeta^{u}}{\Gamma_{2,k}(mu + n)},\quad \zeta,m,n\in \mathbb{T},~k\in \mathbb{R}^+,
\label{eq:3.1}
\end{equation}
where $|\operatorname{Im_j}(m)|<\operatorname{Re}(m),~|\operatorname{Im_j}(n)|<\operatorname{Re}(n).$
They  established the corresponding k-Riemann–Liouville fractional calculus within the bicomplex setting for the extended function. As special cases, we can get $\mathbb{E}_{k,m,1}(\zeta)=\mathbb{E}_{k,m}(\zeta)$, $\mathbb{E}_{1,m,n}(\zeta)=\mathbb{E}_{m,n}(\zeta)$ and $\mathbb{E}_{1,m,1}(\zeta)=\mathbb{E}_{m}(\zeta)$.

In the papers \cite{bakhet2025new,bakhet2025bicomplex}, Bakhet et al.  introduced a new deformation parameter k through the bicomplex 
k-Gamma function, thereby generalizing the earlier bicomplex Mittag-Leffler functions into a scalable analytic family. Their work forms a natural continuation of the formulations by Sharma et al. \cite{ragarwal2021sept} and Agrawal et al. \cite{ragarwal2022}, which respectively defined the one-parameter and two-parameter Mittag-Leffler functions in the bicomplex domain using the bicomplex Gamma function, thus giving rise to a hierarchy of bicomplex fractional models.

In a recent article, Sharma et al. \cite{ragarwal2025} introduced the bicomplex Prabhakar function  using the bicomplex gamma function \cite{goyal2006}. Also discussed its properties and derived its  bicomplex Laplace and Mellin transforms. The bicomplex three-parameter  ML function (Prabhakar function ) is defined as 
\begin{equation}\label{eq: bc ml 3 para}
\mathbb{E}_{m,n}^l(\zeta)=\sum_{u=0}^{\infty}\frac{(l)_u}{u!\Gamma_2(m u+n )}\zeta^u,  \end{equation}
where $\zeta, m,n,l\in \mathbb{T},$ 
$\zeta= z_1+i_2z_2= \zeta_1e_1+\zeta_2e_2$, $m= m_1e_1+m_2e_2,  ~ n=n_1e_1+n_2e_2$ with $|\operatorname{Im_j}(m)|<\operatorname{Re}(m),~|\operatorname{Im_j}(n)|<\operatorname{Re}(n).$	Here $(l)_u$ is bicomplex pochhammer symbol defined  as \cite{goyal2006}
\[
(\ell)_u = \ell(\ell+1)(\ell+2)\cdots(\ell+u-1) = \frac{\Gamma_2(\ell+u)}{\Gamma_2(\ell)}.
\]

Kumar et al. \cite{akumar2011a} introduced the bicomplex Laplace transform and its basic properties. These findings have a lot of potential in the area of signal processing.
These results can be highly applicable in the field of signal processing.\\
\begin{definition}[Bicomplex	  Laplace transform]
	Let $f(t)$ be a bicomplex valued function of exponential order $M\in \mathbb{R}$. Then Laplace
	transform of $f(t)$ for $t \ge 0$  defined as:
	\begin{equation}\label{eq:bc lap def}
	\mathcal{L}(f(t);\xi) =\tilde{f}(\xi)= \int_{0}^{\infty}f(t)e^{-\xi t}dt.
	\end{equation}
	Here $\tilde{f}(\xi)$ exist and is convergent for all $\xi \in  \Omega$
	where
		\begin{equation}\label{eq:lap domain}\notag
	\Omega= \{\xi:H_\rho(\xi)~\text{ represents a Right half plane}~ a_0>	M+|a_3|  \},
	\end{equation}	
   and  $\xi  = a_0 + a_1i_1 + a_2i_2 + a_3j.$\\
 \end{definition}  
\noindent   For more details about bicomplex Laplace transform see \cite{akumar2011a,ragarwal2014a}.  
Agrawal et al. \cite{ragarwal2014a} discussed the convolution theorem for Laplace transform as follows
\begin{theorem}

Let $\mathcal{L}((f(t);\xi)=\tilde{f}(\xi)$ and $\mathcal{L}((g(t);\xi)= \tilde{g}(\xi)$ with 
$ 
\operatorname{Re}(\xi )>M +|\operatorname{Im_j}(\xi)| $ 

where $M=\max(M_1,M_2) $ and $f(t)$ and $g(t)$ are  of exponential order $M_1$ and $M_2$ respectively. Then  
\[
\mathcal{L}( (f * g)(t);\xi)  = \tilde{f}~\tilde{g}.
\]
\end{theorem}
Bicomplex
Laplace transform of the Prabhakar function	is given by \cite{ragarwal2025}
\begin{equation}\label{eq:bclt 3para}
\mathcal{L}(t^{n-1}\mathbb{E}_{ m,n}^l(r t^{ m});\xi)=\int_{0}^{\infty}e^{-\xi t}\mathbb{E}_{ m, n}^l(r t^{ m})t^{n-1}dt=\frac{\xi^{ ml-n}}{(\xi^{ m}-r)^l},
\end{equation}\label{eq:plap}
where	
$|\operatorname{Im_j}(n)|<\operatorname{Re}(n),~|\operatorname{Im_j}(\xi)|<\operatorname{Re}(\xi),~|r\xi^{- m}|_j\prec1.$\\

\noindent The generalisation of integration and differentiation operators in fractional calculus has been the subject of numerous studies in recent years.
    As fractional calculus continues to demonstrate its utility in modeling real-world phenomena-from anomalous diffusion to long-memory time series—the bicomplex extension developed in this work provides a natural generalization for systems where multiple coupled variables interact through non-local, memory-dependent relationships. The future development of this theory promises to reveal new mathematical structures and practical applications that bridge classical analysis, fractional calculus, and complex systems science.

The introduction of bicomplex Prabhakar derivatives opens a new chapter in fractional calculus. The rich algebraic structure of bicomplex numbers, combined with the flexibility of Prabhakar kernels, offers a versatile framework applicable across diverse scientific and engineering disciplines.

The Prabhakar derivative is a generalized operator in fractional calculus, introduced as an extension of the classical Riemann–Liouville and Caputo derivatives \cite{kilbas2004generalized}. It is based on the Prabhakar function  and is useful for describing systems with memory and hereditary properties. The Riemann-Liouville integral, derivative, Caputo fractional derivative, and other definitions of fractional integrals and derivatives can be found in the literature (see, e.g. \cite{miller1993introduction,kibas1993,kilbas2006}).

\noindent In this paper we have studied the   fractional calculus involving  bicomplex Prabhakar function  \cite{ragarwal2025}. We have derived the  bicomplex extension of the  Prabhakar  integral operator, Prabhakar derivative and regularized Prabhakar derivative along with their properties.
We have obtained the Laplace transform of the regularized Prabhakar fractional derivative, which provides a powerful and systematic tool for handling fractional initial value problems. This transform greatly simplifies the analysis of memory-dependent processes and enables the effective treatment of complex fractional differential equations. Using this approach, we solve the Cauchy problem involving Prabhakar-type fractional operators, demonstrating both the applicability and efficiency of the method. These results underscore the utility of the regularized Prabhakar derivative in modeling a wide class of fractional dynamic systems with non-local memory effects.

\section {Fractional calculus of bicomplex Prabhakar function}

\noindent 
Here, we study the  compositions of bicomplex Riemann–Liouville fractional integral $\mathbb{I}_{a+}^{\alpha}$ and derivative $\mathbb{D}_{a+}^{\alpha} $ \cite{goswami2022riemann,goswami2021generalization} associated with the bicomplex Prabhakar function.
\begin{theorem}\label{th:frac pr}
	Let $a \in \mathbb{R}_{+} = [0, \infty)$, and let $\alpha, m, n, l, r \in \mathbb{T}$ with $|\operatorname{Im_j}(\upalpha)|<\operatorname{Re}(\upalpha),~|\operatorname{Im_j}(m )|<\operatorname{Re}(m),~|\operatorname{Im_j}(n )|<\operatorname{Re}(n).$
	 Then for $x > a$, the following formulas hold:

	\begin{enumerate}
	\item 
	$	\mathbb{I}_{a+}^{\alpha} \big[(t - a)^{n - 1} \mathbb{E}_{m, n}^{l}\big(r (t - a)^{m}\big)\big])(x)
	=(x - a)^{n + \alpha - 1} \mathbb{E}_{m, n + \alpha}^{l}\big(r (x - a)^{m})\big..$ 
	\item 
$	\mathbb{D}_{a+}^{\alpha} \big[(t - a)^{n - 1} \mathbb{E}_{m, n}^{l}\big(r (t - a)^{m}\big)\big])(x)
	=(x - a)^{n - \alpha - 1} \mathbb{E}_{m, n - \alpha}^{l}\big(r (x - a)^{m})\big..$
\end{enumerate}

\end{theorem}
\begin{proof}
	(1). By \cite[Theorem 3, equation (3.1)]{kilbas2004generalized} complex Prabhakar function satisfies for  \(\operatorname{Re}(\alpha_i) > 0,~\operatorname{Re}(m_i) > 0, ~\operatorname{Re}(n_i) > 0\), (i=1,2)
    \begin{equation}
	(\mathbb{I}_{a+}^{\alpha_i} \big[(t - a)^{n_i - 1} \mathbb{E}_{m_i, n_i}^{l_i}\big(r_i (t - a)^{m_i}\big)\big])(x)
			=(x - a)^{n_i + \alpha_i - 1} \mathbb{E}_{m_i, n_i + \alpha_i}^{l_i}\big(r_i (x - a)^{m_i})\big.
				\end{equation}
    By using the idempotent representation, 
    we have
	\begin{equation}
	\begin{split}
		\mathbb{I}_{a+}^{\alpha} \big[(t - a)^{n - 1} \mathbb{E}_{m, n}^{l}\big(r (t - a)^{m}\big)\big])(x)
		=&	(\mathbb{I}_{a+}^{\alpha_1} \big[(t - a)^{n_1 - 1} \mathbb{E}_{m_1, n_1}^{l_1}\big(r_1 (t - a)^{m_1}\big)\big])(x)e_1\\
		&+	(\mathbb{I}_{a+}^{\alpha_2} \big[(t - a)^{n_2 - 1} \mathbb{E}_{m_2, n_2}^{l_2}\big(r_2 (t - a)^{m_2}\big)\big])(x)e_2\\
			=&(x - a)^{n_1 + \alpha_1 - 1} \mathbb{E}_{m_1, n_1 + \alpha_1}^{l_1}\big(r_1 (x - a)^{m_1})\big.e_1\\
			&+(x - a)^{n_2 + \alpha_2 - 1} \mathbb{E}_{m_2, n_2 + \alpha_2}^{l_2}\big(r_2 (x - a)^{m_2})\big.e_2\\
=&	(x - a)^{n + \alpha - 1} \mathbb{E}_{m, n + \alpha}^{l}\big(r (x - a)^{m})\big..
	\end{split}
	\end{equation}
	(2). By \cite[Theorem 3, equation (3.2)]{kilbas2004generalized} complex Prabhakar function satisfies for  \(\operatorname{Re}(\alpha_i) > 0,~\operatorname{Re}(m_i) > 0, ~\operatorname{Re}(n_i) > 0\), (i=1,2) 
    \begin{equation}
        \mathbb{D}_{a+}^{\alpha_i} \big[(t - a)^{n_i - 1} \mathbb{E}_{m_i, n_i}^{l_i}\big(r_i (t - a)^{m_i}\big)\big])(x)=(x - a)^{n_i - \alpha_i - 1} \mathbb{E}_{m_i, n_i - \alpha_i}^{l_i}\big(r_i (x - a)^{m_i})\big.
    \end{equation}
By using the idempotent representation  
       \begin{equation}
	\begin{split}
		\mathbb{D}_{a+}^{\alpha} \big[(t - a)^{n - 1} \mathbb{E}_{m, n}^{l}\big(r (t - a)^{m}\big)\big])(x)=& \mathbb{D}_{a+}^{\alpha_1} \big[(t - a)^{n_1 - 1} \mathbb{E}_{m_1, n_1}^{l_1}\big(r_1 (t - a)^{m_1}\big)\big])(x)e_1\\
		&+\mathbb{D}_{a+}^{\alpha_2} \big[(t - a)^{n_2 - 1} \mathbb{E}_{m_2, n_2}^{l_2}\big(r_2 (t - a)^{m_2}\big)\big])(x)e_2\\
	=&(x - a)^{n_1 - \alpha_1 - 1} \mathbb{E}_{m_1, n_1 - \alpha_1}^{l_1}\big(r_1 (x - a)^{m_1})\big.e_1\\
	&+(x - a)^{n_2 - \alpha_2 - 1} \mathbb{E}_{m_2, n_2 - \alpha_2}^{l_2}\big(r_2 (x - a)^{m_2})\big.e_2\\
	=&(x - a)^{n - \alpha - 1} \mathbb{E}_{m, n - \alpha}^{l}\big(r (x - a)^{m})\big..
	\end{split}
	\end{equation}
\end{proof}
\noindent  Convolution identity is fundamental in fractional calculus with Prabhakar operators, since it allows composition laws and semi group properties of fractional integrals/derivatives defined with these kernels. 
Convolution identity for bicomplex Prabhakar function is given by the following result.
\begin{theorem}[Convolution Identity]
	Let $m, n, l, \nu, \sigma, r \in \mathbb{T}$ with $|\operatorname{Im_j}(m) |<\operatorname{Re}(m),|\operatorname{Im_j}(n )|<\operatorname{Re}(n),|\operatorname{Im_j}(\nu)|<\operatorname{Re}(\nu).$ 
    Then
	\begin{equation}\label{eq: result}
	\int_0^x (x - t)^{n - 1} \mathbb{E}_{m,n}^{l}\big(r(x - t)^{m}\big) \, t^{\nu - 1} \mathbb{E}_{m,\nu}^{\sigma}\big(r t^{m}\big) \, dt = x^{n + \nu - 1} \mathbb{E}_{m,n + \nu}^{l + \sigma}\big(r x^{m}\big).
	\end{equation}
	
\end{theorem}

\begin{proof}	
	By using idempotent representation, we have
	\begin{equation}\label{eq:idempotent}
\begin{split}
&	\int_0^x (x - t)^{n - 1} \mathbb{E}_{m,n}^{l}\big(r(x - t)^{m}\big) \, t^{\nu - 1} \mathbb{E}_{m,\nu}^{\sigma}\big(r t^{m}\big) \, dt \\&=	\int_0^x (x - t)^{n_1 - 1} \mathbb{E}_{m_1,n_1}^{l_1}\big(r_1(x - t)^{m_1}\big) \, t^{\nu_1 - 1} \mathbb{E}_{m_1,\nu_1}^{\sigma_1}\big(r_1 t^{m_1}\big) \, dt~ e_1\\
&~+	\int_0^x (x - t)^{n_2 - 1} \mathbb{E}_{m_2,n_2}^{l_2}\big(r_2(x - t)^{m_2}\big) \, t^{\nu_2 - 1} \mathbb{E}_{m_2,\nu_2}^{\sigma_2}\big(r_2 t^{m_2}\big) \, dt ~e_2.
\end{split}
	\end{equation}
     
 By using the Parseval theorem for the Laplace transform component-wise for both  idempotent components \cite{titchmarsh1937introduction,kilbas2004generalized} and Laplace transform of  Prabhakar function (see, e.g. \cite{gorenflo2014,kilbas2006}), we obtain for $|r_1 \xi_1^{-m_1}| < 1$		\begin{equation}\label{eq:lapla-ML}
		\begin{split}
		&	\mathcal{L}	\left(  \int_0^x (x - t)^{n_1 - 1} \mathbb{E}_{m_1,n_1}^{l_1}\big(r_1(x - t)^{m_1}\big) \, t^{\nu_1 - 1} \mathbb{E}_{m_1,\nu_1}^{\sigma_1}\big(r_1 t^{m_1}\big) \, dt;\xi_1\right)  \\
		&=\mathcal{L} \left( x^{n_1 - 1} \mathbb{E}^{l_1}_{m_1,n_1}(r_1 x^{m_1});\xi_1 \right) \cdot \mathcal{L} \left( x^{\nu_1 - 1} \mathbb{E}^{\sigma_1}_{m_1,\nu_1}(r_1 x^{m_1});\xi_1 \right)\\
		&=\xi_1^{-n_1} \left(1 - r_1 \xi_1^{-m_1} \right)^{-l_1} \cdot \xi_1^{-\nu_1} \left(1 - r_1 \xi_1^{-m_1} \right)^{-\sigma_1}\\
	&	= \xi_1^{-(n_1 + \nu_1)} \left(1 - r_1 \xi_1^{-m_1} \right)^{-(l_1 + \sigma_1)}.\\
		\end{split}
		\end{equation}			
\noindent The RHS of \eqref{eq:lapla-ML} is the Laplace transform of \( x^{n_1 + \nu_1 - 1} \mathbb{E}^{l_1 + \sigma_1}_{m_1, n_1 + \nu_1} (r_1 x^{m_1)} \), hence by taking the inverse Laplace transform, we obtain 
	\begin{equation}\label{eq: 1}
		\left[ \int_0^x (x - t)^{n_1 - 1} \mathbb{E}_{m_1,n_1}^{l_1}\big(r_1(x - t)^{m_1}\big) \, t^{\nu_1 - 1} \mathbb{E}_{m_1,\nu_1}^{\sigma_1}\big(r_1 t^{m_1}\big)  dt\right] =  x^{n_1 + \nu_1 - 1} \mathbb{E}^{l_1 + \sigma_1}_{m_1, n_1 + \nu_1} (r_1 x^{m_1}).
	\end{equation}
   	Similarly, we get for $|r_2 \xi_2^{-m_2}| < 1,$	
\begin{equation}\label{eq: 2}
	\left[ \int_0^x (x - t)^{n_2 - 1} \mathbb{E}_{m_2,n_2}^{l_2}\big(r_2(x - t)^{m_2}\big) \, t^{\nu_2 - 1} \mathbb{E}_{m_2,\nu_2}^{\sigma_2}\big(r_2 t^{m_2}\big)  dt\right] =  x^{n_2 + \nu_2 - 1} \mathbb{E}^{l_2 + \sigma_2}_{m_2, n_2 + \nu_2} (r_2 x^{m_2}).
	\end{equation}
	By using equations \eqref{eq: 1} and \eqref{eq: 2} in the equation \eqref{eq:idempotent}, we get
		\begin{equation}
	\begin{split}
	&	\int_0^x (x - t)^{n - 1} \mathbb{E}_{m,n}^{l}\big(r(x - t)^{m}\big) \, t^{\nu - 1} \mathbb{E}_{m,\nu}^{\sigma}\big(r t^{m}\big) \, dt \\
	&=x^{n_1 + \nu_1 - 1} \mathbb{E}^{l_1 + \sigma_1}_{m_1, n_1 + \nu_1} (r_1 x^{m_1})~ e_1+ x^{n_2 + \nu_2 - 1} \mathbb{E}^{l_2 + \sigma_2}_{m_2, n_2 + \nu_2} (r_2 x^{m_2})~e_2\\
	&= x^{n + \nu - 1} \mathbb{E}^{l + \sigma}_{m, n + \nu} (r x^{m}),
	\end{split}
	\end{equation}
where $|r \xi^{-m}|_j\prec1,  $  from \eqref{eq:plap}.
\end{proof}

\section{Bicomplex  Prabhakar fractional operators}
Motivated by the recent  works 
\cite{toksoy2024geometrical,toksoy2024hyperbolic,dubey2014note} following definitions  can be used component-wise for the bicomplex valued functions as follows:
\begin{definition}\label{def:lebesgur}
Let	the space $L(a, b)$ of Lebesgue measurable 	functions (see, e.g. \cite{kilbas2004generalized}) on a finite interval $[a, b]$ of the real line $\mathbb{R}$ is defined as 
	\begin{equation}
	L(a,b) =\left\lbrace f(x): ~  \parallel f \parallel _1=\int_{a}^{b}|f(x)|dx<\infty \right\rbrace .
	\end{equation}
       Then the corresponding  the space of bicomplex-valued functions $L_\mathbb{T}(a,b)$ comprises of all functions of the type 	\(f = f_1 e_1 + f_2 e_2\), where $f_1,f_2 \in L(a,b).$
\end{definition}
\begin{definition}
   The space $W^{m,1}(a,b)$  is the Sobolev
space (see, e.g. 	\cite{polito2016some}) defined as
\begin{equation}
    W^{m,1}(a,b)=\left\lbrace f\in L^1(a,b):\frac{d^m}{dt^m}f\in L^1(a,b)\right\rbrace. 
\end{equation}
 Then the corresponding  the space of bicomplex-valued functions $W_{\mathbb{T}}^{m,1}(a,b)$ comprises of all functions of the type 	\(f = f_1 e_1 + f_2 e_2\), where $f_1,f_2 \in W^{m,1}(a,b).$
\end{definition}

\begin{definition}
	Let the space  $AC^M(a,b)$ is (see, e.g. 	\cite{polito2016some})  the space of real-valued functions $f(t)$ with 
	continuous derivatives up to order $M-1$ on $(a,b)$ such that $f^{ (M-1)} (t)$ belongs to the
	space of absolutely continuous functions $AC(a,b)$. 
    Then the corresponding  the space of bicomplex-valued functions $AC_\mathbb{T}^M(a,b)$ comprises of all functions of the type 	\(f = f_1 e_1 + f_2 e_2\), where $f_1,f_2 \in AC^M(a,b).$
    \end{definition}


\noindent Following lemma plays a crucial role in proving the subsequent results. By establishing this lemma, we ensure that the subsequent theorems can be derived more systematically and with greater clarity.
\begin{lemma}
	If $m \in \mathbb{T}$ where $m=m_1e_1+m_2e_2$, if	$\operatorname{Re}(m_1)>0$ and  $\operatorname{Re}(m_2)  >0$ then  $|\operatorname{Im_j}(m)|<\operatorname{Re}(m).$ Also if $\operatorname{Re}(m_1)>k$ and  $\operatorname{Re}(m_2)  >k$ then   $|\operatorname{Im_j}(m)|<\operatorname{Re}(m)-k.$
	
\end{lemma}
\begin{proof}
	
	Let $m \in \mathbb{T}$ in terms of real components, we have
	\begin{equation}
	\begin{split}
	m&= p_0+i_1p_1+ i_2p_2+i_1i_2p_3\\
	&=(p_0+i_1p_1) +i_2( p_2+i_1p_3).\\
	\end{split}
	\end{equation}
    By the definition of idempotent representation give in the equation \eqref{eq:idempotent def}, we get
    \begin{equation}
        m=m_1e_1+m_2e_2,
    \end{equation}
	where $m_1= (p_0+p_3) +i_1(p_1-p_2),~m_2= (p_0-p_3) +i_1(p_1+p_2)$
 and
    $\operatorname{Re}(m_1)= p_0+p_3, ~\operatorname{Re}(m_2)= p_0-p_3. $ \\
    Also, $p_0=\dfrac{\operatorname{Re}(m_1)+\operatorname{Re}(m_2)}{2},~p_3=\dfrac{\operatorname{Re}(m_1)-\operatorname{Re}(m_2)}{2}$\\

	\noindent Since
	$\operatorname{Re}(m_1)>0$ and  $\operatorname{Re}(m_2)  >0$
	\begin{eqnarray}
	&\Rightarrow & p_0+p_3 >0 ~\text{and}~ p_0-p_3 >0.\\ \notag 
    &\Rightarrow & p_0 >-p_3 ~\text{and}~ p_0>p_3.\\ \notag 
     &\Rightarrow &-p_3 < p_0 ~\text{and}~ p_3< p_0 .\\ \notag 	
	&	\Rightarrow & ~ |p_3|<p_0.\label{eq:re(alpha)g0}\\ \notag 
	&\Rightarrow&|\operatorname{Im_j}(m)|<\operatorname{Re}(m).
	\end{eqnarray}
	
\noindent Also, if 
	$\operatorname{Re}(m_1)>k$ and  $\operatorname{Re}(m_2)  >k,$ where $k\in \mathbb{N}$
	\begin{eqnarray}
	&\Rightarrow & p_0+p_3 >k ~\text{and}~ p_0-p_3 >k.\\ \notag     
	&\Rightarrow & p_0-k >-p_3 ~\text{and}~ p_0-k >p_3.\\ \notag 
	&	\Rightarrow& ~ |p_3|<p_0-k.\label{eq:re(alpha)g01}\\ \notag 
	&\Rightarrow&|\operatorname{Im_j}(m)|<\operatorname{Re}(m)-k.
	\end{eqnarray}

\end{proof}

\subsection{Bicomplex Prabhakar kernel}
The Prabhakar kernel is a generalized kernel function defined through the Prabhakar function. It forms the basis of the Prabhakar fractional integral and derivative operators.
In order to discuss bicomplex Prabhakar integral, we define 
the Prabhakar kernel in bicomplex space. 
\begin{theorem}
		Let   $t\in \mathbb{R},~m,n,r,l \in \mathbb{T},~m= m_1e_1+m_2e_2,  ~ n=n_1e_1+n_2e_2,~ r=r_1e_1+r_2e_2,~l=l_1e_1+l_2e_2$ and $|\operatorname{Im_j}(m)|<\operatorname{Re}(m),~|\operatorname{Im_j}(n)|<\operatorname{Re}(n).$ Then   bicomplex Prabhakar kernel involving bicomplex Prabhakar function \eqref{eq: bc ml 3 para} is defined as 
	\begin{equation}\label{eq: be p kernel}
	e_{m,n,r}^{l}(t)=t^{n-1}(\mathbb{E}_{m,n}^l)(r t^{m}).
	\end{equation}
	
\end{theorem}
\begin{proof}
	
Consider
	\begin{equation}
e_{m,n,r}^{l}(t)=t^{n-1}(\mathbb{E}_{m,n}^l)(r t^{m}).
\end{equation}
By using idempotent representation
\begin{equation}
\begin{split}
e_{m,n,r}^{l}(t)
&=t^{n_1-1}(\mathbb{E}_{m_1,n_1}^{l_1})(r_1 t^{m_1}) e_1+t^{n_2-1}(\mathbb{E}_{m_2,n_2}^{l_2})(r_2 t^{m_2})e_2\\
&=e_{m_1,n_1,r_1}^{l_1}(t)e_1+e_{m_2,n_2,r_2}^{l_2}(t)e_2.\\
\end{split}
\end{equation}
Now,
\begin{equation}\label{eq:newfun1}
e_{m_1,n_1,r_1}^{l_1}(t)=t^{n_1-1}(\mathbb{E}_{m_1,n_1}^{l_1})(r_1 t^{m_1}),
\end{equation}
is the  complex Prabhakar kernel  exists for $\operatorname{Re}(m_1)>0,~\operatorname{Re}(n_1)>0$ \cite{polito2016some,garra2014hilfer}.  \\

\noindent Similarly,
\begin{equation}\label{eq:newfun2}
e_{m_2,n_2,r_2}^{l_2}(t)=t^{n_2-1}(\mathbb{E}_{m_2,n_2}^{l_2})(r_2 t^{m_2}),
\end{equation}
is also  the  complex Prabhakar kernel  exists for $\operatorname{Re}(m_2)>0,~\operatorname{Re}(n_2)>0$ (see, for e.g. \cite{polito2016some,garra2014hilfer}).

	Since $e_{m_1,n_1,r_1}^{l_1}(t)$ and $e_{m_2,n_2,r_2}^{l_2}(t)$ are exist in $T_1$ and $ T_2$ respectively, by the Ringleb
decomposition theorem  (see, for e.g. \cite{ringlab1933}),  \eqref{eq: be p kernel} is also exists in $T$ for  $|\operatorname{Im_j}(m)|<\operatorname{Re}(m),~|\operatorname{Im_j}(n)|<\operatorname{Re}(n).$\\
\end{proof}


\begin{theorem}
	Let $f\in L_{\mathbb{T}}(a,b),~0\le a<t<b\le\infty$  is bicomplex valued function and  $m,n,r,l \in \mathbb{T},~m= m_1e_1+m_2e_2,  ~ n=n_1e_1+n_2e_2,~ r=r_1e_1+r_2e_2,~l=l_1e_1+l_2e_2$ and $|\operatorname{Im_j}(m)|<\operatorname{Re}(m),~|\operatorname{Im_j}(n)|<\operatorname{Re}(n).$ Then   bicomplex Prabhakar integral  is defined as 
\begin{equation}\label{eq:bc p int1}	(\mathbb{E}_{m,n,r,a^+}^l)f(t)=(f*e_{m,n,r}^l)(t)=\int_{\varUpsilon}(t-y)^{n-1}\mathbb{E}_{m,n}^l[{r(t-y)^m}]f(y)dy,
	\end{equation}
	where $\varUpsilon =(\gamma_1, \gamma_2)$ is a piecewise smooth  curve in $\mathbb{T}$ and 	$\gamma_1:a $ to $t$ and $\gamma_2:a $ to $t$ are piecewise smooth curves in $\mathbb{C}.$
	\end{theorem}
\begin{proof} Consider
	\begin{equation}	(\mathbb{E}_{m,n,r,a^+}^l)f(t)=(f*e_{m,n,r}^l)(t)=\int_{\varUpsilon}(t-y)^{n-1}\mathbb{E}_{m,n}^l[{r(t-y)^m}]f(y)dy.
	\end{equation}
    By using idempotent representation
	\begin{equation}
	\begin{split}		
(\mathbb{E}_{m,n,r,a^+}^l)f(t)=&(f*e_{m,n,r}^l)(t)\\
	=&\int_{\varUpsilon}(t-y)^{n-1}\mathbb{E}_{m,n}^l[{r(t-y)^m}]f(y)dy\\
	=&\int_{\gamma_1}(t-y)^{n_1-1}\mathbb{E}_{m_1,n_1}^{l_1}[{r_1(t-y)^{m_1}}]f_1(y)dy~e_1\\
	&+\int_{\gamma_2}(t-y)^{n_2-1}\mathbb{E}_{m_2,n_2}^{l_2}[{r_2(t-y)^{m_2}}]f_2(y)dy~e_2\\
	=&(f_1*e_{m_1,n_1,r_1}^{l_1})(t)~e_1+(f_2*e_{m_2,n_2,r_2}^{l_2})(t)~e_2\\=&	(\mathbb{E}_{m_1,n_1,r_1,a^+}^{l_1})f_1(t)~e_1+(\mathbb{E}_{m_2,n_2,r_2,a^+}^{l_2})f_2(t)~e_2.\label{eq:bc p int idem}
	\end{split}
	\end{equation}
Now, 	\begin{equation}	(\mathbb{E}_{m_1,n_1,r_1,a^+}^{l_1})f_1(t)=(f_1*e_{m_1,n_1,r_1}^{l_1})(t)=\int_{\gamma_1}(t-y)^{n_1-1}\mathbb{E}_{m_1,n_1}^{l_1}[{r_1(t-y)^{m_1}}]f_1(y)dy,
	\end{equation}
	is the complex Prabhakar integral \cite{polito2016some,garra2014hilfer} exists for $\operatorname{Re}(m_1)>0,~\operatorname{Re}(n_1)>0.$ \\
	Similarly,
	\begin{equation}	(\mathbb{E}_{m_2,n_2,r_2,a^+}^{l_2})f_2(t)=(f_2*e_{m_2,n_2,r_2}^{l_2})(t)=\int_{\gamma_2}(t-y)^{n_2-1}\mathbb{E}_{m_2,n_2}^{l_2}[{r_2(t-y)^{m_2}}]f_2(y)dy,
	\end{equation}
	is also complex Prabhakar integral exists for $\operatorname{Re}(m_2)>0,~\operatorname{Re}(n_2)>0.$ \\
	Since $ (\mathbb{E}_{m_1,n_1,r_1,a^+}^{l_1})f_1(t)$ and $(\mathbb{E}_{m_2,n_2,r_2,a^+}^{l_2})f_2(t)$ are exist in $T_1$ and $ T_2$ respectively, by the Ringleb
	decomposition theorem,  \eqref{eq:bc p int1} is also exists in $T$ for $|\operatorname{Im_j}(m)|<\operatorname{Re}(m),~|\operatorname{Im_j}(n)|<\operatorname{Re}(n).$	
\end{proof}

\noindent Following result shows that the operator $\mathbb{E}_{m,n,r,a^+}^l$  is bounded on the space $L_{\mathbb{T}}(a,b).$ 

If the domain of definition itself is a curve in the complex plane, say $f : \Gamma \to \mathbb{C}$, where $\Gamma$ is a contour (path) in $\mathbb{C}$, then one can define an 
$L^p$ norm \emph{along that contour}.

We parametrize the contour by $z = \gamma(t), \quad t \in [a,b]$, and define
\[\parallel f\parallel _{L^p(\Gamma)} = 
\left( \int_a^b |f(\gamma(t))|^p \, |\gamma'(t)| \, dt \right)^{1/p}.\]
Here, $|\gamma'(t)|\,dt$ represents the \emph{arc-length measure} along the contour $\Gamma$.  
This is analogous to integrating along a path, but the $L^p$ norm depends only on the modulus $|f|$, not on the direction or orientation of the contour.

\begin{theorem}\label{th:bounded}
	Let $m,n,r,l \in \mathbb{T},~m= m_1e_1+m_2e_2,  ~ n=n_1e_1+n_2e_2,~ r=r_1e_1+r_2e_2,l=l_1e_1+l_2e_2$ and $|\operatorname{Im_j}(m)|<\operatorname{Re}(m),~|\operatorname{Im_j}(n)|<\operatorname{Re}(n)$  and $b > a$, then the operator 
$	(\mathbb{E}_{m,n,r,a^+}^l)$ 	is bounded on $L_{\mathbb{T}}(a, b)$ 

\begin{equation}\label{eq:bound}
\int_{\varUpsilon}\left| \mathbb{E}_{m,n,r,a^+}^lf(t)\right|_jdt \prec K\int_{\varUpsilon}\left|f(y)\right|_jdy,\quad f \in L_{\mathbb{T}}(a,b),
\end{equation}
  	where $\varUpsilon =(\gamma_1, \gamma_2)$, represents a pair of paths in \(\mathbb{T}\), with \(\gamma_1\) and \(\gamma_2\) being piecewise smooth paths in \(\mathbb{C}\) from \(a\) to \(t\).  $ K=K_1e_1+K_2e_2,$ and

$\displaystyle K_i=(b-a)^{\operatorname{Re}(n_i)}\sum_{k=0}^{\infty}\tfrac{|(l_i)_k|}{|\Gamma_2(m_i k+n_i )|[\operatorname{Re}(m_i) k+\operatorname{Re}(n_i)]}\frac{|r_i (b-a)^{\operatorname{Re}(m_i)}|^k}{k!}\quad  (i=1,2).$
\end{theorem}
\begin{proof}
The result can be obtained by breaking up LHS of \eqref{eq:bound} into idempotent components and then using the result \cite[p.40]{kilbas2004generalized} for
complex domain therein. Combining the idempotent components again, we get the desired
 result in the bicomplex domain.

The constant $K$ depends on $b-a$, the parameters $m, n, l, r$,  
and converges for $|r(b-a)^{Re(m)}| <$ radius determined by the Mittag-Leffler function.  
For small intervals or appropriate parameter choices, $K$ can be made arbitrarily small.
\end{proof}
\noindent The Prabhakar derivative (both in the Riemann–Liouville sense and in the regularized Caputo-type sense) plays a central role in extending fractional calculus. Now we introduce the bicomplex Prabhakar derivative as follows:
\begin{theorem}
	Let $f\in L_{\mathbb{T}}(a,b),~0\le a<t<b\le\infty$ is bicomplex valued function and $ f*e^{-l}_{m,k-n,r}(.)\in W_{\mathbb{T}}^{k,1}(a,b),~k=\lceil n \rceil=\lceil\operatorname{Re} (n_1)\rceil e_1+\lceil \operatorname{Re}(n_2)\rceil e_2\in \mathbb{D}.$ The  bicomplex Prabhakar derivative  is defined as 
	\begin{equation}\label{eq:bc prabhakar der}
	(\mathbb{D}_{m,n,r,a^+}^l f)(t)=\left(\frac{d^k}{dt^k}(\mathbb{E}_{m,k-n,r,a^+}^{-l}f) \right) (t),
	\end{equation}
	where $m,n,r,l \in \mathbb{T},~m= m_1e_1+m_2e_2,  ~ n=n_1e_1+n_2e_2,~ r=r_1e_1+r_2e_2,l=l_1e_1+l_2e_2$ and $|\operatorname{Im_j}(m)|<\operatorname{Re}(m),~|\operatorname{Im_j}(n)|<\operatorname{Re}(n).$
	
\end{theorem}
\begin{proof}
    Consider 
	\begin{equation}
	(\mathbb{D}_{m,n,r,a^+}^l f)(t)=\left(\frac{d^m}{dt^m}(\mathbb{E}_{m,k-n,r,a^+}^{-l}f) \right) (t).
	\end{equation}
	By using the idempotent representation
		\begin{equation}
	\begin{split}
		(\mathbb{D}_{m,n,r,a^+}^l f)(t)=&
	\left(\frac{d^k}{dt^k}(\mathbb{E}_{m_1,k-n_1,r_1,a^+}^{-l_1}f_1) \right) (t)~e_1+	\left(\frac{d^m}{dt^m}(\mathbb{E}_{m_2,k-n_2,r_2,a^+}^{-l_2}f_2) \right) (t)~e_2	\\
	=&(\mathbb{D}_{m_1,n_1,r_1,a^+}^{l_1} f_1)(t)~e_1+(\mathbb{D}_{m_2,n_2,r_2,a^+}^{l_2} f_2)(t)~e_2.
	\end{split}
	\end{equation}
	Now, 	\begin{equation}
(\mathbb{D}_{m_1,n_1,r_1,a^+}^{l_1}f_1)(t)=\left(\frac{d^k}{dt^k}(\mathbb{E}_{m_1,k-n_1,r_1,a^+}^{-l_1}f_1) \right) (t),
	\end{equation}
	is the complex Prabhakar derivative  \cite{polito2016some,garra2014hilfer} exists for $\operatorname{Re}(m_1)>0,~\operatorname{Re}(n_1)>0,~k=\lceil \operatorname{Re}(n_1)\rceil  .$ \\
	Similarly, 	
	\begin{equation}	(\mathbb{D}_{m_2,n_2,r_2,a^+}^{l_2} f_2)(t)=\left(\frac{d^k}{dt^k}(\mathbb{E}_{m_2,k-n_2,r_2,a^+}^{-l_2}f_2) \right) (t),\end{equation}	
	is also complex Prabhakar derivative exists for $\operatorname{Re}(m_2)>0,~\operatorname{Re}(n_2)>0,~k=\lceil \operatorname{Re}( n_2)\rceil. $ 
    
	Since $(\mathbb{D}_{m_1,n_1,r_1,a^+}^{l_1} f_1)(t)$ and $(\mathbb{D}_{m_2,n_2,r_2,a^+}^{l_2} f_2)(t)$ are exist in $T_1$ and $ T_2$ respectively, by the Ringleb
	decomposition theorem (see, e.g. \cite{ragarwal2022bicomplex}),  \eqref{eq:bc prabhakar der} is also exists in $T$ where $|\operatorname{Im_j}(m)|<\operatorname{Re}(m),~|\operatorname{Im_j}(n)|<\operatorname{Re}(n).$
     
\end{proof}

Similarly, the bicomplex  regularized Prabhakar derivative can be defined as in the following theorem:
\begin{theorem}
	Let $f \in AC_\mathbb{T}^M(a,b)$ be a bicomplex valued function with $0\le a<t<b\le\infty,$ 
continuous derivatives up to order n-1 on (a,b) Let $m,n,r,l \in \mathbb{T}$, $m= m_1e_1+m_2e_2$, $ n=n_1e_1+n_2e_2$,  $r=r_1e_1+r_2e_2,~l=l_1e_1+l_2e_2,$ the bicomplex  regularized Prabhakar derivative  is defined as
	\begin{equation}\label{eq: bc reg der}
	\begin{split}
	(^C\mathbb{D}_{m,n,r,a^+}^l f)(t)&=\left(\mathbb{E}_{m,k-n,r,a^+}^{-l}\frac{d^k}{dt^k}f \right) (t)\\
	&=(\mathbb{D}_{m,n,r,a^+}^l f)(t)-\sum_{p=0}^{k-1}t^{p-n}\mathbb{E}_{m,p-n+1}^{-l}(r t^m)f^{(p)}(a^+).
	\end{split}
	\end{equation}
\end{theorem}
\begin{proof}
	Consider 
	\begin{equation}
	\begin{split}
	(^C\mathbb{D}_{m,n,r,a^+}^l f)(t)&=\left(\mathbb{E}_{m,k-n,r,a^+}^{-l}\frac{d^k}{dt^k}f \right) (t)\\
	&=(\mathbb{D}_{m,n,r,a^+}^l f)(t)-\sum_{p=0}^{k-1}t^{p-n}\mathbb{E}_{m,p-n+1}^{-l}(r t^m)f^{(p)}(a^+).
	\end{split}
	\end{equation}
	By using the idempotent representation
\begin{equation}
\begin{split}
\left( ^C\mathbb{D}_{m,n,r,a^+}^l f\right) (t)	=&\left(\mathbb{E}_{m_1,k-n_1,r_1,a^+}^{-l_1}\frac{d^k}{dt^k}f_1 \right) (t)e_1+\left(\mathbb{E}_{m_2,k-n_2,r_2,a^+}^{-l_2}\frac{d^k}{dt^k}f_2 \right) (t)e_2\\
=&\left( (\mathbb{D}_{m_1,n_1,r_1,a^+}^{l_1} f_1)(t)-\sum_{p=0}^{k-1}t^{p-n_1}\mathbb{E}_{m_1,p-n_1+1}^{-l_1}(r_1 t^{m_1})f_1^{(p)}(a^+)\right) e_1\\
&+\left( (\mathbb{D}_{m_2,n_2,r_2,a^+}^{l_2} f_2)(t)-\sum_{p=0}^{k-1}t^{p-n_2}\mathbb{E}_{m_2,p-n_2+1}^{-l_2}(r_2 t^{m_2})f_2^{(p)}(a^+)\right) e_2\\
=&\left( ^C\mathbb{D}_{m_1,n_1,r_1,a^+}^{l_1} f_1\right) (t) e_1+\left( ^C\mathbb{D}_{m_2,n_2,r_2,a^+}^{l_2} f_2\right) (t)e_2.\\
\end{split}
\end{equation}
Now, 	\begin{equation}
	\begin{split}
	\left( ^C\mathbb{D}_{m_1,n_1,r_1,a^+}^{l_1} f_1\right) (t)&=\left(\mathbb{E}_{m_1,k-n_1,r_1,a^+}^{-l_1}\frac{d^k}{dt^k}f_1 \right) (t)\\
	&=(\mathbb{D}_{m_1,n_1,r_1,a^+}^{l_1} f_1)(t)-\sum_{p=0}^{k-1}t^{p-n_1}\mathbb{E}_{m_1,p-n_1+1}^{-l_1}(r_1 t^{m_1})f_1^{(p)}(a^+),
	\end{split}
	\end{equation}
is the complex regularized  Prabhakar  derivative exists for $\operatorname{Re}(m_1)>0,~\operatorname{Re}(n_1)>0.$ \\
	Similarly, 	
	 	\begin{equation}
	\begin{split}
	\left( ^C\mathbb{D}_{m_2,n_2,r_2,a^+}^{l_2} f_2\right) (t)&=\left(\mathbb{E}_{m_2,k-n_2,r_2,a^+}^{-l_2}\frac{d^k}{dt^k}f_2 \right) (t)\\
	&=(\mathbb{D}_{m_2,n_2,r_2,a^+}^{l_2} f_2)(t)-\sum_{p=0}^{k-1}t^{p-n_2}\mathbb{E}_{m_2,p-n_2+1}^{-l_2}(r_2 t^{m_2})f_2^{(p)}(a^+),
	\end{split}
	\end{equation}
	is also complex regularized  Prabhakar derivative exists for $\operatorname{Re}(m_2)>0,~\operatorname{Re}(n_2)>0.$ 
    
	Since $(^C\mathbb{D}_{m_1,n_1,r_1,a^+}^{l_1} f_1)(t)$ and $(^C\mathbb{D}_{m_2,n_2,r_2,a^+}^{l_2} f_2)(t)$ are exists in $T_1$ and $ T_2$ respectively, by the Ringleb
	decomposition theorem, equation \eqref{eq: bc reg der} is also exists in $T$ where for 
	$|\operatorname{Im_j}(m)|<\operatorname{Re}(m)$ and 
	$\operatorname{Im_j}(n)|<\operatorname{Re}(n).$
	\end{proof}
\begin{remark}
\noindent
for $l=0$ the operator \eqref{eq:bc p int1} coincides with the Riemann-Liouville fractional integral  of order $n$ in the bicomplex space
\begin{equation}\label{eq:bc relation}
(\mathbb{E}_{m,n,r,a^+}^0 f)\equiv \mathbb{I}_{a+}^n f.
\end{equation}
For $l=0$ the operator \eqref{eq: bc reg der} coincides with the Caputo derivative.
\end{remark}

\section{Some properties of bicomplex  Prabhakar fractional operators}
Theorem \ref{th:learity1}, prove the linearity property of the bicomplex Prabhakar integral operator, bicomplex Prabhakar derivative and bicomplex regularized Prabhakar derivative. 
\begin{theorem}\label{th:learity1}
Let $f,g$ 
are  bicomplex valued functions and $c,d$ be  constants. Let $m,n,r,l \in \mathbb{T}$ 
and 
$|\operatorname{Im_j}(m)|<\operatorname{Re}(m),~|\operatorname{Im_j}(n)|<\operatorname{Re}(n).$ For (i) and (ii)  $f,g\in L_{\mathbb{T}}(a,b)$, and for (iii)  $f,g\in AC_\mathbb{T}^M(a,b).$
  Then  
\begin{eqnarray}
&	(i)~ &	(\mathbb{E}_{m,n,r,a^+}^l)(c f(t)+dg(t)) =	c(\mathbb{E}_{m,n,r,a^+}^l) f(t)+	d(\mathbb{E}_{m,n,r,a^+}^l)g(t).\\
&  (ii)~	&(\mathbb{D}_{m,n,r,a^+}^l )(cf(t)+dg(t))=c(\mathbb{D}_{m,n,r,a^+}^l )f(t)+d(\mathbb{D}_{m,n,r,a^+}^l )g(t).  \\
& (iii)~	&(^C\mathbb{D}_{m,n,r,a^+}^l )(cf(t)+dg(t))=c(^C\mathbb{D}_{m,n,r,a^+}^l )f(t)+d(^C\mathbb{D}_{m,n,r,a^+}^l )g(t). 
\end{eqnarray}

\end{theorem}
\begin{proof}(i)
By using \eqref{eq:bc p int idem}, we have
\begin{equation}
\begin{split}
&(\mathbb{E}_{m,n,r,a^+}^l)( cf(t)+dg(t)) \\
&=\left[  (\mathbb{E}_{m_1,n_1,r_1,a^+}^{l_1})e_1+	(\mathbb{E}_{m_2,n_2,r_2,a^+}^{l_2})e_2\right] \left[\{cf_1(t)+dg_1(t)\} e_1+\{cf_2(t)+dg_2(t)e_2\} \right] \\
&=\left[ (\mathbb{E}_{m_1,n_1,r_1,a^+}^{l_1})\{cf_1(t)+dg_1(t)\}\right] e_1+\left[ (\mathbb{E}_{m_2,n_2,r_2,a^+}^{l_2})\{cf_2(t)+dg_2(t)\}\right] e_2\\
&=\left[c\{ (\mathbb{E}_{m_1,n_1,r_1,a^+}^{l_1})e_1+(\mathbb{E}_{m_2,n_2,r_2,a^+}^{l_2})e_2\}\{f_1(t)e_1+f_2(t)e_2\}\right] \\
&~~+\left[d\{ (\mathbb{E}_{m_1,n_1,r_1,a^+}^{l_1})e_1+(\mathbb{E}_{m_2,n_2,r_2,a^+}^{l_2})e_2\}\{g_1(t)e_1+g_2(t)e_2\}\right] \\
&=c(\mathbb{E}_{m,n,r,a^+}^l) f(t)+	d(\mathbb{E}_{m,n,r,a^+}^l)g(t).
\end{split}
\end{equation}

\noindent In a similar way, we prove the linearity property for bicomplex Prabhakar derivative and bicomplex  regularized Prabhakar derivative.
\end{proof}

\noindent The semigroup property of the bicomplex Prabhakar integral operator is important because it guarantees  computational simplification, while also establishing the operator as a natural extension of classical fractional integrals. Here we obtain the semi group property for  bicomplex integral operator as follows: 
\begin{theorem}[Semi group property ]
	let $\sigma ,\nu ,m,n,l,r\in \mathbb{T},$  then for $|\operatorname{Im_j}(m )|<\operatorname{Re}(m),~|\operatorname{Im_j}(n )|<\operatorname{Re}(n),~|\operatorname{Im_j}(\nu)|<\operatorname{Re}(\nu)$   the  relation  
	\begin{equation}\label{eq: comp1}
	\mathbb{E}^{l}_{m,n,r,a^+}\mathbb{E}^{\sigma}_{m,\nu ,r,a^+}	f\equiv\mathbb{E}^{l+\sigma}_{m,n+\nu  ,r,a^+}f,
	\end{equation}
	is valid for any bicomplex valued summable function $f \in  L_{\mathbb{T}}(a, b).$ In particular,
	\begin{equation}\label{eq: comp2}
	\mathbb{E}^{l}_{m,n,r,a^+}\mathbb{E}^{-l}_{m,\nu ,r,a^+}	f\equiv\mathbb{I}^{n+\nu }_{a^+}f.
	\end{equation}
\end{theorem}
\begin{proof}	
By using the equation  \eqref{eq:bc p int idem} and   \cite[Theorem 8, equation (6.1)]{kilbas2004generalized},
we obtain
\begin{equation}
\begin{split}
\left( \mathbb{E}^{l}_{m,n,r,a^+} \mathbb{E}^{\sigma}_{m,\nu,r,a^+}f \right)(x)=&\left( \mathbb{E}^{l_1}_{m_1,n_1,r_1,a^+} \mathbb{E}^{\sigma_1}_{m_1,\nu_1,r_1,a^+} f_1\right)(x)e_1+\left( \mathbb{E}^{l_2}_{m_2,n_2,r_2,a^+} \mathbb{E}^{\sigma_2}_{m_2,\nu_2,r_2,a^+} f_2 \right)(x)e_2\\
=& \left( \mathbb{E}^{l_1 + \sigma_1}_{m_1,n_1 + \nu_1,r_1,a^+}f_1 \right)(x)e_1+\left( \mathbb{E}^{l_2 + \sigma_2}_{m_2,n_2 + \nu_2,r_2,a^+} f_2 \right)(x)e_2\\
&=\mathbb{E}^{l+\sigma}_{m,n+\nu  ,r,a^+}f.
\end{split}
\end{equation}
\noindent This proves the relation
\begin{equation}\label{eq:semigroup11}
    \mathbb{E}^{l}_{m,n,r,a^+} \, \mathbb{E}^{\sigma}_{m,\nu,r,a^+} f \equiv\ \mathbb{E}^{l + \sigma}_{m,n + \nu,r,a^+} f.
\end{equation}

By substituting $\sigma=-l$  in the above equation \eqref{eq:semigroup11}, we obtain
\begin{equation}\label{eq:semigroup12}
    \mathbb{E}^{l}_{m,n,r,a^+} \, \mathbb{E}^{-l}_{m,\nu,r,a^+} f \equiv\ \mathbb{E}^{0}_{m,n + \nu,r,a^+} f.
\end{equation}
By using equation (\ref{eq:bc relation}), the operator at right side in the equation \eqref{eq:semigroup12}, coincides with the Riemann-Liouville fractional integral  of order $n+\nu$ in the bicomplex space. Hence, we get
\begin{equation}	\mathbb{E}^{l}_{m,n,r,a^+}\mathbb{E}^{-l}_{m,\nu ,r,a^+}	f\equiv\mathbb{I}^{n+\nu }_{a^+}f.
	\end{equation}
\end{proof}

\noindent The study of compositions of fractional calculus operators with the bicomplex Prabhakar
integral operator is very important. They
provide analytical simplification by replacing successive applications of
fractional operators with a single Prabhakar operator having modified
parameters.
Here, we obtain  the compositions of fractional calculus operators with 
integral operator with bicomplex  Prabhakar function in the kernel  as follows:

\begin{theorem}[Compositions of Riemann–Liouville fractional integral operator with Prabhakar integral operator]
	Let $f\in  L_{\mathbb{T}}(a, b)$ be a bicomplex valued function. Let $\alpha,m,n,l,r\in \mathbb{T},~ ~\alpha= \alpha_1e_1+\alpha_2e_2,  ~ \upbeta=m_1e_1+m_2e_2,~ n=n_1e_1+n_2e_2,l=l_1e_1+l_2e_2,$ then for $|\operatorname{Im_j}(\upalpha)|<\operatorname{Re}(\upalpha),~|\operatorname{Im_j}(m) |<\operatorname{Re}(m),~|\operatorname{Im_j}(n) |<\operatorname{Re}(n),$
	\begin{equation}
\mathbb{I}^{\alpha}_{a^+}\mathbb{E}^{l}_{m,n,r,a^+}f\equiv\mathbb{E}^{l}_{m,n+\alpha ,r,a^+}f\equiv\mathbb{E}^{l}_{m,n,r,a^+}	\mathbb{I}^{\alpha}_{a^+}f.
	\end{equation}
\end{theorem}
\begin{proof}
	From equation  \eqref{eq:bc p int idem} and   equation  \cite[Theorem 6,(5.1)]{kilbas1995fractional} we obtain
    for $x > a$ 
    \begin{equation}
    \begin{split}
    (	\mathbb{I}^{\alpha}_{a+}\mathbb{E}^{l}_{m,n,r,a^+}f)(x)&=(	\mathbb{I}^{\alpha_1}_{a+}\mathbb{E}^{l_1}_{m_1,n_1,r_1,a^+}~f_1)(x)~e_1+(	\mathbb{I}^{\alpha_2}_{a^+}\mathbb{E}^{l_2}_{m_2,n_2,r_2,a^+}~f_2)(x)~e_2\\
    &=\mathbb{E}^{l_1}_{m_1,n_1+\alpha_1 ,r_1,a^+}~f_1(x)~e_1+\mathbb{E}^{l_2}_{m_2,n_2+\alpha_2 ,r_2,a^+}~f_2(x)~e_2\\
    &=\mathbb{E}^{l}_{m,n+\alpha ,r,a^+}~f(x).\\
    \end{split}
    \end{equation}
    Also,
	\begin{equation}
    \begin{split}
    (	\mathbb{E}^{l}_{m,n,r,a^+}\mathbb{I}^{\alpha}_{a^+}f)(x)&=(	\mathbb{E}^{l_1}_{m_1,n_1,r_1,a^+}\mathbb{I}^{\alpha_1}_{a^+}~f_1)(x)~e_1+(	\mathbb{E}^{l_2}_{m_2,n_2,r_2,a^+}\mathbb{I}^{\alpha_2}_{a^+}~f_2)(x)~e_2\\
    &=\mathbb{E}^{l_1}_{m_1,n_1+\alpha_1 ,r_1,a^+}~f_1(x)~e_1+\mathbb{E}^{l_2}_{m_2,n_2+\alpha_2 ,r_2,a^+}~f_2(x)~e_2\\
    &=\mathbb{E}^{l}_{m,n+\alpha ,r,a^+}~f(x).\\
    \end{split}
    \end{equation}
    \end{proof}
\noindent This is the commutative property of  the Riemann–Liouville fractional integral operator and  bicomplex Prabhakar integral operator.
Similarly  the result for the Riemann–Liouville fractional derivative operator $\mathbb{D}^{\alpha}_{a+}$ as  \cite{kilbas2004generalized} follows:
\begin{theorem}[Compositions of Riemann–Liouville fractional derivative operator with Prabhakar
integral operator]
	Let $f\in  L_{\mathbb{T}}(a, b)$ be a bicomplex valued function. Let $\alpha,m,n,l,r\in \mathbb{T},~ ~\alpha= \alpha_1e_1+\alpha_2e_2,  ~ \upbeta=m_1e_1+m_2e_2,~ n=n_1e_1+n_2e_2,l=l_1e_1+l_2e_2,$ then for $|\operatorname{Im_j}(\upalpha)|<\operatorname{Re}(\upalpha),~|\operatorname{Im_j}(m )|<\operatorname{Re}(m),~|\operatorname{Im_j}(n) |<\operatorname{Re}(n),$
	\begin{equation}
\mathbb{D}^{\alpha}_{a^+}\mathbb{E}^{l}_{m,n,r,a^+}f\equiv\mathbb{E}^{l}_{m,n-\alpha ,r,a^+}f,
\end{equation}
holds for any continuous function $f\in  C(a, b).$ In particular, for $k\in \mathbb{N}$ and  $|\operatorname{Im_j}(n)|<\operatorname{Re}(n)-k$
\begin{equation}
\left( \frac{d}{dx}\right)^k \mathbb{E}^{l}_{m,n,r,a^+}f\equiv\mathbb{E}^{l}_{m,n-k,r,a^+}f.
\end{equation}\end{theorem}

\noindent Here we construct the left inversion operator to the  operator $\mathbb{E}^{l}_{m,n,r,a+}.$
\begin{theorem}[Left Inversion Theorem]
	Let $ m,n,l,r\in\mathbb{T}, ~|\operatorname{Im_j}(m )|<\operatorname{Re}(m),~|\operatorname{Im_j}(n) |<\operatorname{Re}(n)$. Then the operator $\mathbb{E}^{l}_{m,n,r,a^+}$ is invertible in the space $L(a, b)$ and for $f =f_1e_1+f_2e_2\in L_{\mathbb{T}}(a, b)$, its left inverse is given by
	\[
	\left( \left[ \mathbb{E}^{l}_{m,n,r,a^+} \right]^{-1} f \right)(x) = \left( \mathbb{D}^{n+\nu}_{a^+} \, \mathbb{E}^{-l}_{m,\nu,r,a^+} f \right)(x),
	\]
	where $\nu \in \mathbb{T}$ and $|\operatorname{Im_j}(\nu) |<\operatorname{Re}(\nu)$.
\end{theorem}
\begin{proof}
	let us first define the operator $\left[ \mathbb{E}^{l}_{m,n,r,a+} \right]^{-1}$ such that 
	\begin{equation}\label{eq: operator}	
	\left( \mathbb{E}^{l}_{m,n,r,a^+} \psi \right)(x) = f(x), \quad\psi =\psi_1e_1+\psi_2e_2\in L_{\mathbb{T}}(a,b).
	\end{equation}

	By the Theorem \ref{th:bounded}, $\mathbb{E}^{l}_{m,n,r,a^+}\psi \in L_{\mathbb{T}}(a, b),$ hence $f \in L_{\mathbb{T}}(a, b)$.

	 Apply the operator $\mathbb{E}^{-l}_{m,\nu,r,a^+}$ to both sides of the equation \eqref{eq: operator}
	\[
	\mathbb{E}^{-l}_{m,\nu,r,a^+} f(x) = \mathbb{E}^{-l}_{m,\nu,r,a^+} \mathbb{E}^{l}_{m,n,r,a^+} \psi(x),
	\]
	where  $\nu \in \mathbb{T}$ and $|\operatorname{Im_j}(\nu) |<\operatorname{Re}(\nu)$.
	
	By using the semi group property  \eqref{eq: comp2}, we get

	\[ 	\mathbb{E}^{-l}_{m,\nu,r,a^+} f(x) = \left( \mathbb{I}^{n+\nu}_{a^+} \psi \right)(x).
	\]
	Now, applying the Riemann-Liouville fractional derivative $\mathbb{D}^{n+\nu}_{a^+}$   and using the  result \cite[Lemma 3, equation (3.9)]{kilbas2004generalized}, we obtain
    \begin{equation}
\begin{split}
	\mathbb{D}^{n+\nu}_{a^+} \mathbb{E}^{-l}_{m,\nu,r,a^+} f(x)&=
		\mathbb{D}^{n_1+\nu_1}_{a^+} \mathbb{E}^{-l_1}_{m_1,\nu_1,r_1,a^+} f_1(x) e_1+	\mathbb{D}^{n_2+\nu_2}_{a^+} \mathbb{E}^{-l_2}_{m_2,\nu_2,r_2,a^+} f_2(x)e_2\\
		&=\mathbb{D}^{n_1+\nu_1}_{a^+} \mathbb{I}^{n_1+\nu_1}_{a^+} \psi_1 (x) e_1+	\mathbb{D}^{n_2+\nu_2}_{a^+}  \mathbb{I}^{n_2+\nu_2}_{a^+} \psi_2 (x)e_2\\
		&=\psi_1(x)e_1+\psi_2(x)e_2\\
        &=\psi(x).
\end{split}
\end{equation}
	Thus, we have recovered $\psi$ from $f$ using the composition of operators:
	\[
	\psi(x) = \left( \mathbb{D}^{n+\nu}_{a+} \mathbb{E}^{-l}_{m,\nu,r,a^+} f \right)(x),
	\]
	which proves that $\left[ \mathbb{E}^{l}_{m,n,r,a^+}\right]^{-1} $ is left inversion of $\mathbb{E}^{l}_{m,n,r,a^+}$ in $L_{\mathbb{T}}(a,b).$ 
\end{proof}

\section{Laplace transform of fractional derivatives and application in solving Cauchy Problem } 
In this section we derive the Laplace transform of the bicomplex Prabhakar  integral, bicomplex Prabhakar  derivative, bicomplex regularized Prabhakar  derivative. The LT of  bicomplex Prabhakar derivative plays a fundamental role in simplifying the analysis of fractional differential equations.  In this section we use  the notation $\tilde{f}(\xi)$ to represent the Laplace transform of the function $f$.   
\begin{theorem}\label{th:lt   p int}
Let $f$ is  bicomplex valued function. For (i),(ii)  $f\in L_{\mathbb{T}}^1(a,b)$, and for (iii)  $f\in AC_\mathbb{T}^M(a,b),$	and $\xi,m,n,l,r \in \mathds{T} $ where $\xi=\xi_1e_1+\xi_2e_2$ $|\operatorname{Im_j}(\xi)|<\operatorname{Re}(\xi),~|r \xi^{-m}|_j\prec1,~ k=\lceil n \rceil=\lceil\operatorname{Re} (n_1)\rceil e_1+\lceil \operatorname{Re}(n_2)\rceil e_2\in \mathbb{D}$. Bicomplex
	LT of  the bicomplex fractional operators are    \\ 
 
(i) $\mathcal{L}\left(\mathbb{E}_{m,n,r,0^+}^lf(t) ;\xi\right)=\mathcal{L}\left( (f*e_{m,n,r}^l)(t) ;\xi\right) =\dfrac{\xi^{ml-n}}{(\xi^{m}-r)^{l}}\tilde{f}(\xi).$\\

(ii)$\mathcal{L}\left( \mathbb{D}^{l}_{m,n,r,0^+} f(t) ;\xi\right)
= \xi^{n - m l} \left(\xi^{m} - r\right)^{l} \tilde{f}(\xi)
- \sum_{p=0}^{k-1} \xi^{\,k-p-1}
\left( \mathbb{E}^{-l}_{m,k-n-p,r,0^+} f \right)(0^{+})$.  \\

(iii)	$\mathcal{L}\left( 	^C\mathbb{D}_{m,n,r,0^+}^l f(t);\xi\right)= \xi^{n-ml}\left(\xi^{m}-r\right)^{l}
	\left[
	\tilde{f}(\xi) - \sum_{p=0}^{k-1} \xi^{-p-1} f^{(p)}(0^{+})
	\right].$

    \end{theorem}
    
\begin{proof}(i)
    With the Laplace transform of  Prabhakar derivative \cite[p.27]{giusti2020practical}
	and the result (\ref{eq:bc p int1}), we have
  \begin{equation}
	\begin{split}
	\mathcal{L}\left( \mathbb{E}_{m,n,r,0^+}^lf(t) ;\xi\right) &= \mathcal{L}\left( \mathbb{E}_{m_1,n_1,r_1,0^+}^{l_1}f_1(t) ;\xi_1\right) e_1+ \mathcal{L}\left(  \mathbb{E}_{m_2,n_2,r_2,0^+}^{l_2}f_2(t) ;\xi_2\right)e_2\\
	&=\mathcal{L}\left(  (f_1*e_{m_1,n_1,r_1}^{l_1})(t) ;\xi_1\right)e_1+\mathcal{L}\left(  (f_2*e_{m_2,n_2,r_2}^{l_2})(t) ;\xi_2\right)e_2\\
	&=\mathcal{L}\left(  e_{m_1,n_1,r_1}^{l_1}(t) ;\xi_1\right)\tilde{f_1}(\xi_1)e_1+\mathcal{L}\left(  e_{m_2,n_2,r_2}^{l_2}(t) ;\xi_2\right)\tilde{f_2}(\xi_2)e_2\\
	&=\frac{\xi_1^{m_1l_1-n_1}}{(\xi_1^{m_1}-r_1)^{l_1}}\tilde{f_1}(\xi_1)e_1+\frac{\xi_2^{m_2l_2-n_2}}{(\xi_2^{m_2}-r_2)^{l_2}}\tilde{f_2}(\xi_2)e_2\\
	&=\frac{\xi^{ml-n}}{(\xi^{m}-r)^{l}}\tilde{f}(\xi),
	\end{split}
	\end{equation}
    for $|\operatorname{Im_j}(\xi)|<\operatorname{Re}(\xi),~|r \xi^{-m}|_j\prec1.$
\end{proof}
  Similarly, (ii) and (iii) can be proved.  

The regularized Prabhakar derivative naturally arises in modeling systems  
with distributed memory kernels, such as viscoelastic materials with hierarchical  
relaxation processes or multi-scale diffusion in heterogeneous media.  
The parameter l controls the degree of memory distribution, with l=0  
recovering the classical Riemann-Liouville case.

Consider the  Cauchy problem involving the regularized Prabhakar derivative. In the following theorem, we find the solution of Cauchy problem in the bicomplex space. 
\begin{theorem}\label{th:cauchygen}
Let $f\in  L_{\mathbb{T}}[a, b] $ be a bicomplex valued function satisfying the Cauchy problem 
\begin{equation} 
\begin{cases}
^{C}\mathbb{D}^{l}_{m,n,r,0^+} \, f(t) = A f(t), \\[6pt]
f^{(k)}(0^{+}) = \tau_{k}, \quad k=0,1,\dots,m-1,
\end{cases}
\end{equation} 
where
$m,n,r,A, \tau_{0},\dots,\tau_{m-1} \in \mathbb{T},|\operatorname{Im_j}(m)|<\operatorname{Re}(m),~|\operatorname{Im_j}(n)|<\operatorname{Re}(n),~k=\lceil n \rceil=\lceil\operatorname{Re} (n_1)\rceil e_1+\lceil \operatorname{Re}(n_2)\rceil e_2,~l=l_1e_1+l_2e_2\in \mathbb{D}$.
\end{theorem}
\begin{proof}
For the Cauchy problem \begin{equation} \label{eq:5.14}
^{C}\mathbb{D}^{l}_{m,n,r,0^+} \, f(t) = A f(t)
\end{equation} 
    	Applying the Laplace transform to \eqref{eq:5.14} gives
	\[
	\xi^{n-ml}(\xi^{m}-r)^{l}
	\left(\tilde{f}(\xi) - \sum_{k=0}^{m-1}\xi^{-k-1}\tau_j\right)
	= A\,\tilde{f}(\xi).
	\]
	Rearranging, we find
	\begin{equation}\label{eq:lcauchy}  
	\tilde{f}(\xi)
	= \frac{\xi^{n-ml}(\xi^{m}-r)^{l}
		\sum_{k=0}^{m-1}\xi^{-k-1}\tau_k}
	{\xi^{n-ml}(\xi^{m}-r)^{l}-A}.	\end{equation}
	We can rewrite \eqref{eq:lcauchy} as
	\[
	\tilde{f}(\xi)
	= \sum_{k=0}^{m-1}\xi^{-k-1}\tau_k
	\cdot \frac{1}{1-\dfrac{A}{\xi ^{n-ml}(\xi^{m}-r)^{l}}}.
	\]
By braking up into idempotent components then for sufficiently large $\operatorname{\xi_i},~(i=1,2)$ we can expand the denominator as a Neumann (geometric) series and then combining the  idempotent components we get the following result

	\[
	\frac{1}{1-\dfrac{A}{\xi^{n-ml}(\xi^{m}-r)^{l}}}
	= \sum_{k=0}^{\infty}A^k\,\xi^{-k(n-ml)}(\xi^{m}-r)^{-kl}, \quad 
\left| \frac{A}{\xi^{\,n - m l} \left( \xi^{m} - r \right)^{l}} \right| \prec 1.	\]
	Substituting into \(\tilde{f}(\xi)\) gives
	\begin{equation}
	\tilde{f}(\xi)
	= \sum_{k=0}^{m-1}\sum_{j=0}^{\infty}
	A^k\,\xi^{-k-1-j(n-ml)}(\xi^{m}-r)^{-jl}\tau_k.
	\label{eq:complex-514}
	\end{equation}

By termwise Laplace inversion, the analytic solution of \eqref{eq:complex-514} in the bicomplex domain is
	\begin{equation}
		f(t)
		= \sum_{k=0}^{m-1}\sum_{j=0}^{\infty}
		A^{j}t^{n j + k}
		\mathbb{E}^{jl}_{m,n j + k +1}(r t^{m})\,\tau_k.
	\end{equation}
    \end{proof}

\begin{theorem}\label{th:cauchy-nonhomo}
    Let $f,g\in  L_{\mathbb{T}}[a, b] $ and of exponential order, be bicomplex valued functions and
\begin{equation}\label{cauchy1}
{}^{C}\mathbb{D}^{l}_{m,n,r,0^+} f(t)+kf(t) = g(t), 
\qquad f(0^+) = \tau_0,
\end{equation}
where $k$ and $\tau_0$ are constants and the parameters satisfy 
$m,n,r \in \mathbb{T},~|\operatorname{Im_j}(m)|<\operatorname{Re}(m),~|\operatorname{Im_j}(n)|<\operatorname{Re}(n)$ and $l\in \mathbb{D},~l=l_1e_1+l_2e_2, ~0<l_1,l_2<1.$

The integral is taken from the lower limit $0^+$. Then
\begin{equation}
f(t)= (h *g)(t)+\tau_{0}\mathcal{L}^{-1}\Big\{\frac{\xi^{n-1}(1 - r \xi^{-m})^{l}}{\xi^{n}(1 - r \xi^{-m})^{l} + k}\Big\}(t),
\end{equation}
where 
\[
h(t) = \mathcal{L}^{-1}\left\{\frac{1}{\xi^{n}\big(1 - r\xi^{-m}\big)^{l} + k}\right\}(t),\]

\end{theorem}	
    \begin{proof}
        	Taking the Laplace transform of \eqref{cauchy1} and using the formula  
	\[
	\mathcal{L}({}^{C}\mathbb{D}^{l}_{m,n,r,0^+} f(t) ;\xi)
	= \xi^{n}(1 - r \xi^{-m})^{l}\tilde{f}(\xi) - \xi^{n-1}(1 - r \xi^{-m})^{l}f(0^+),
	\]
	we obtain
	\[
	\xi^{n}(1 - r \xi^{-m})^{l}\tilde{f}(\xi) - \xi^{n-1}(1 - r \xi^{-m})^{l}\tau_0 + k \tilde{f}(\xi) = \tilde{g}(\xi).
	\]
 Rearranging gives
	\[
	\tilde{f}(\xi)[\xi^{n}(1 - r \xi^{-m})^{l} + k]
	= \xi^{n-1}(1 - r \xi^{-m})^{l}\tau_0 + \tilde{g}(\xi).
	\]
    \[
H(\xi) := \frac{1}{\xi^{n}\big(1 - r\xi^{-m}\big)^{l} + k}.
\]

Then we have
\[
\widetilde{f}(\xi)
= H(\xi)\,\widetilde{g}(\xi)
+ \tau_{0}\,\xi^{n-1}\big(1 - r\xi^{-m}\big)^{l}H(\xi).
\]

Taking the inverse Laplace transform termwise gives
\[
f(t)= (h *g)(t)+\tau_{0}\mathcal{L}^{-1}\Big\{\frac{\xi^{n-1}(1 - r \xi^{-m})^{l}}{\xi^{n}(1 - r \xi^{-m})^{l} + k}\Big\}(t),
\]
where, $h(t) = \mathcal{L}^{-1}\{H(\xi)\}(t),
~
g(t) = \mathcal{L}^{-1}\{\widetilde{g}(\xi)\}(t),~
(h * g)(t) = \int_{0}^{t} h(t - s)\,g(s)\,ds.
$

    \end{proof}

 For $k=0$ in the above theorem, $h(t)=[t^{n-1}\mathbb{E}_{ m,n}^l(r t^{ m})]$ and we obtain the following corollary as a special case.   
\begin{corollary}\label{th:cauchy-nonhomocor}
    Let $f,g\in  L_{\mathbb{T}}[a, b] $ and of exponential order, be bicomplex valued functions and
\begin{equation}\label{cauchy}
{}^{C}\mathbb{D}^{l}_{m,n,r,0^+} f(t) = g(t), 
\qquad f(0^+) = \tau_0,
\end{equation}
where the parameters satisfy 
$m,n,r \in \mathbb{T},~|\operatorname{Im_j}(m)|<\operatorname{Re}(m),~|\operatorname{Im_j}(n)|<\operatorname{Re}(n)$ and $l\in \mathbb{D},~l=l_1e_1+l_2e_2, ~0<l_1,l_2<1.$

The integral is taken from the lower limit $0^+$. Then
\begin{equation}
{f}(t) = \tau_0 +[t^{n-1}\mathbb{E}_{ m,n}^l(r t^{ m})]*g(t).
\end{equation}

\end{corollary}	




\section{Conclusion}
In this paper, bicomplex extension of the  Prabhakar  integral operator, Prabhakar derivative and regularized Prabhakar derivative are derived along with their properties. The Riemann–Liouville fractional integral and differential operators' compositions are determined using $\mathbb{E}^{l}_{m,n,r,a+}.$   Also fractional calculus of bicomplex Prabhakar function is studied.

In this work, we have derived the Laplace transform of the regularized Prabhakar fractional derivative to solve Cauchy problems, establishing a robust framework for handling fractional initial value problems. 
The results demonstrate the efficacy and applicability of this approach in modeling a wide class of fractional dynamic systems with non-local memory effects.  

Several promising avenues for future investigation emerge from this work. This framework can be extended to bicomplex k-parameter Prabhakar functions, enabling the solution of more generalized Cauchy-type fractional differential equations in multidimensional or coupled fractional systems. The results may be generalized to multicomplex spaces of dimension $2^n$ for $n > 2$, providing frameworks for higher-dimensional coupled systems.

We hope this work inspires mathematicians, physicists, and engineers to explore the rich landscape of bicomplex fractional calculus and contribute to its theoretical development and practical applications.

\section*{Declaration}
The authors declare that no funds, grants, or other support were received during the preparation of this manuscript.

\section*{Conflict of interest}
The authors affirm that there are no conflicts of interest related to this work.

\section*{Availability of data}
This study does not involve the use of any external data.

\section*{Authors’ contributions}
U.P.S.: Formal analysis and investigation, Writing - original draft, Conceptualization, Methodology;
R.A.: Conceptualization, Methodology, Writing - review and editing, Resources. All authors read and approved the final manuscript.

\bibliographystyle{apalike}
\bibliography{lref}

@book{kilbas2006,
  author    = {A. A. Kilbas and H. M. Srivastava and J. J. Trujillo},
  title     = {Theory and Applications of Fractional Differential Equations},
  publisher = {Elsevier},
  address   = {Amsterdam},
  year      = {2006}
}

@article{kumar2023generalization,
  author    = {R. Kumar and D. Kumar and S. K. Sharma},
  title     = {Generalization of Riemann-Liouville fractional operators in bicomplex space with applications to fractional differential equations},
  journal   = {Mathematics and Statistics},
  volume    = {11},
  number    = {5},
  pages     = {823--837},
  year      = {2023}
}

@article{thirumalai2025perception,
  author    = {S. Thirumalai and A. Muthunagai},
  title     = {A perception of the applications of bicomplex fractional {L}aplace transform},
  journal   = {Journal of Mathematics and Computer Science},
  volume    = {33},
  pages     = {294--311},
  year      = {2024}
}

@article{prabhakar1971singular,
  author    = {T. R. Prabhakar},
  title     = {A singular integral equation with a generalized Mittag-Leffler function in the kernel},
  journal   = {Yokohama Mathematical Journal},
  volume    = {19},
  pages     = {7--15},
  year      = {1971}
}

@book{titchmarsh1937introduction,
	title={Introduction to the theory of {F}ourier integral},
	author={Titchmarsh, Edward Charles},
	year={1937},
	publisher={The Clarendon Press}
}

@article{riley1953,
	author={J. D. Riley},
	year={1953},
	title={Contributions to the theory of functions of a bicomplex variable},
	number={2},
	pages={132-165},
	journal={ Tohoku Mathematical Journal},
	volume={5}
}

@article{csegre1892,
	author={C. Segre},
	year={1892},
	title={Le rappresentazioni reale delle forme complessee  {G}li {E}nti {I}peralgebrici},
	volume={40},
	number={},
	pages={413-467},
	journal={Math. Ann.},
}

@article{ringlab1933,
	author={ FRIEDRICH Ringleb},
	year={1933},
	title={Beitr\"{a}ge zur Funktionentheorie in hyperkomplexen Systemen},
	volume={57},
	pages={311-340},
	journal={I. Rendiconti del Circolo Matematico di Palermo},
}

@book{price1991,
	author={G. B. Price},
	year={1991},
	title={An Introduction to Multicomplex Spaces and Functions},
	publisher={Marcel Dekker Inc. New York.},
}

@book{kibas1993,
	author={Stefan G. Samko and Anatoly A. Kilbas and Oleg I. Marichev},
	year={1993},
	title={Fractional Integrals and 
 Derivatives:
	Theory and Applications},
	publisher={Gordon and Breach Science, Yverdon, Switzerland },
}

@book{miller1993introduction,
	title={An Introduction to the Fractional Calculus and Fractional Differential Equations},
	author={Miller, Kenneth S and Ross, Bertram},
	year={1993},
	publisher={Wiley, New York}
}

@article{kilbas1995fractional,
	title={Fractional integrals and derivatives of {M}ittag-{L}effler type functions},
	author={Kilbas, AA and Saigo, M},
	journal={Doklady Akademii Nauk Belarusi},
	volume={39},
	number={4},
	pages={22--26 (Russian)},
	year={1995},
}

@article{ronn2001,
	author={Stefan R\"{o}nn},
	year={2001},
	title={Bicomplex algebra and function theory},
	number={},
	pages={1-71},
	journal={arXiv:0101200v1 [Math.CV]},
	volume={},
}

@article{kilbas2004generalized,
	title={Generalized {M}ittag-{L}effler function and generalized fractional calculus operators},
	author={Kilbas, Anatoly A and Saigo, Megumi and Saxena, Ram K},
	journal={Integral Transforms and Special Functions},
	volume={15},
	number={1},
	pages={31--49},
	year={2004},
	publisher={Taylor \& Francis}
}

@article{rochon2004b,
	author={D. Rochon and M. Shapiro},
	year={2004},
	title={On algebraic properties of bicomplex and hyperbolic numbers},
	number={},
	pages={71-110},
	journal={Analele Universitatii din Oradea. Fascicola Matematica},
	volume={11},
}

@article{goyal2006,
	author={S. P. Goyal and Trilok Mathur And Ritu Goyal},
	year={2006},
	title={	Bicomplex gamma And beta Function },
	number={1},
	pages={131-142},
	journal={Journal of Rajasthan Academy Physical Sciences},
	volume={5},
}

@article{akumar2011a,
	author={A. Kumar and P. Kumar},
	year={2011},
	title={Bicomplex Version of {L}aplace Transform},
	number={3},
	pages={225-232},
	journal={International Journal of Engineering and Technology},
	volume={3}
}

@article{meluna2012,
	author={M. Elena Luna-Elizarrar\'{a}s and M. Shapiro and Daniele C. Struppa and A. Vajiac},
	year={2012},
	number={2},
	title={Bicomplex Numbers and their Elementary Functions},
	pages={61-80},
	journal={Cubo A Mathematical Journal},
	volume={14},
}

@article{ragarwal2014a,
	author={Ritu Agarwal and Mahesh Puri Goswami and Ravi P. Agarwal},
	year={2014},
	title={ CONVOLUTION THEOREM AND APPLICATIONS OF
	
	BICOMPLEX {L}APLACE TRANSFORM},
	number={1},	
	pages={113-127},
	journal={Advances in Mathematical
	Sciences and Applications
	},
	volume={24},
}

@article{garra2014hilfer,
  title={{H}ilfer-{P}rabhakar derivatives and some applications},
  author={Garra, Roberto and Gorenflo, Rudolf and Polito, Federico and Tomovski, {\v{Z}}ivorad},
  journal={Applied mathematics and computation},
  volume={242},
  pages={576--589},
  year={2014},
  publisher={Elsevier}
}

@book{gorenflo2014,
	author={Rudolf Gorenflo and Anatoly A. Kilbas and	Francesco Mainardi and Sergei V. Rogosin},
	year={2014},
	title={Mittag-{L}effler functions, Related Topics
	and Application},
	publisher={ Springer,  Berlin Heidelberg },
}

@article{polito2016some,
  title={Some Properties of {P}rabhakar-type Fractional Calculus Operators},
  author={Polito, Federico and Tomovski, Zivorad},
  journal={Fractional Differential Calculus},
  volume={6},
  number={1},
  pages={73--94},
  year={2016}
}

@article{G,
	author={Rudolf Gorenflo and Francesco Mainardi and Sergei V. Rogosin},
	year={2009},
	title={ Mittag-{L}effler Function: Properties and Applications},
	number={},
	pages={269-296},
	journal= {In Handbook of Fractional Calculus with Applications, Volume 1: Basic Theory 	},
	volume={A. Kochubei, Yu.Luchko Berlin/Boston. Series edited by J. A.Tenreiro Machado, },
}

@article{goswami2022riemann,
  title={{R}iemann--{L}iouville fractional operators of bicomplex order and its properties},
  author={Goswami, Mahesh Puri and Kumar, Raj},
  journal={Mathematical Methods in the Applied Sciences},
  volume={45},
  number={10},
  pages={5699--5720},
  year={2022},
  publisher={Wiley Online Library}
}

@article{goswami2021generalization,
  title={Generalization of {R}iemann-{L}iouville Fractional Operators in Bicomplex Space and Applications},
  author={Goswami, Mahesh Puri and Kumar, Raj},
  journal={Mathematics and Statistics},
volume={11},
  number={5},
  pages={802-815},
  year={2023}
}

@incollection{agarwal2023bicomplex,
	title={Bicomplex {M}ittag-{L}effler  function and applications in integral transform and fractional calculus},
	author={Agarwal, Ritu and Sharma, Urvashi Purohit},
	booktitle={Mathematical and Computational Intelligence to Socio-scientific Analytics and Applications},
	pages={157-167},
	year={2023},
	publisher={Springer}
}

@article{giusti2020practical,
	title={A practical guide to {P}rabhakar fractional calculus},
	author={Giusti, Andrea and Colombaro, Ivano and Garra, Roberto and Garrappa, Roberto and Polito, Federico and Popolizio, Marina and Mainardi, Francesco},
	journal={Fractional Calculus and Applied Analysis},
	volume={23},
	number={1},
	pages={9--54},
	year={2020},
	publisher={De Gruyter}
}

@article{ragarwal2022,
	author={Ritu Agarwal and Urvashi Purohit Sharma and Ravi P. Agarwal},
	year={2022},
	title={ Bicomplex {M}ittag-{L}effler Function and associated properties},
	number={},	
	pages={48-60},
	journal={  Journal of Nonlinear Sciences and Applications},
	volume={15},
}

@article{ragarwal2021sept,
	author={Urvashi Purohit Sharma and Ritu Agarwal  and  Kottakkaran  Sooppy Nisar},
	year={2022},
	title={ Bicomplex Two Parameter {M}ittag-{L}effler
	Function and Properties with application to the
	fractional time wave equation},
	number={1},	
	pages={462-481.},
	journal={ Palistine Journal of Mathematics },
	volume={12},
}

@article{ragarwal2022sept,
	author={  Urvashi Purohit Sharma and Ritu Agarwal},
	year={2022},
	title={Bicomplex {L}aplace Transform of Fractional Order, Properties and applications},
	number={1},	
	pages={370-385},
	journal={Journal of Computational Analysis and Applications },
	volume={30},
}

@article{ragarwal2022bicomplex,
	title={Bicomplex {L}andau and {I}kehara Theorems for the {D}irichlet Series},
	author={Agarwal, Ritu and Sharma, Urvashi Purohit and Agarwal, Ravi P and Suthar, Daya Lal and Purohit, Sunil Dutt},
	journal={Journal of Mathematics},
	volume={2022},
	pages={1-8},
	year={2022},
	publisher={Hindawi}
}

@inproceedings{ragarwal2023solution,
	title={Solution of Bicomplex Time Fractional {S}chr{\"o}dinger Equation Involving Bicomplex {M}ittag-{L}effler Function},
	author={Agarwal, Ritu and Sharma, Urvashi P and Agarwal, Ravi P},
	booktitle={International Conference on Mathematical Modelling, Applied Analysis and Computation},
	pages={14--30},
	year={2023},
	organization={Springer}
}

@article{toksoy2024geometrical,
  title={On geometrical characteristics and inequalities of new bicomplex {L}ebesgue Spaces with hyperbolic-valued norm},
  author={Toksoy, Erdem and Sa{\u{g}}{\i}r, Birsen},
  journal={Georgian Mathematical Journal},
  volume={31},
  number={3},
  pages={483--495},
  year={2024},
  publisher={De Gruyter}
}

@article{toksoy2024hyperbolic,
  title={Hyperbolic valued {B}eurling weighted bicomplex {L}ebesgue spaces and some of their geometric characteristics},
  author={Toksoy, Erdem},
  journal={Journal of Inequalities and Applications},
  volume={2024},
  number={1},
  pages={153},
  year={2024},
  publisher={Springer}
}

@article{dubey2014note,
  title={A note on bicomplex {O}rlicz spaces},
  author={Dubey, S and Kumar, R and Sharma, K},
  journal={arXiv preprint arXiv:1401.7112},
  pages={1--12},
  year={2014}
}

@article{ragarwal2025,
	author={Urvashi Purohit Sharma and Ritu Agarwal  },
	year={2025},
	title={ {P}rabhakar  function and unified fractional kinetic equation in bicomplex space },
	number={},	
	pages={},
	journal={ Submitted },
	volume={},
}

@article{bakhet2025new,
  title={On a new version of bicomplex  {M}ittag-{L}effler functions and their applications in fractional kinetic equations},
  author={Bakhet, Ahmed and Zayed, Mohra and Saleem, Mohammed A and Fathi, Mohamed},
  journal={Alexandria Engineering Journal},
  volume={125},
  pages={409-423},
  year={2025},
  publisher={Elsevier}
}

@article{bakhet2025bicomplex,
  title={Bicomplex k-{M}ittag-{L}effler Functions with Two Parameters: Theory and Applications to Fractional Kinetic Equations},
  author={Bakhet, Ahmed and Hussain, Shahid and Zayed, Mohra and Fathi, Mohamed},
  journal={Fractal and Fractional},
  volume={9},
  number={6},
  pages={344},
  year={2025},
  publisher={MDPI}
}

@article{zayed2025k,
  title={On the k-Bicomplex {G}amma and k-Bicomplex {B}eta Functions and Their Properties},
  author={Zayed, Mohra and Bakhet, Ahmed and Fathi, Mohamed and Saleem, Mohammed A},
  journal={Journal of Mathematics},
  volume={2025},
  number={1},
  pages={2657169},
  year={2025},
  publisher={Wiley Online Library}
}

\end{document}